\documentclass[a4paper, 11pt]{amsart}
\usepackage[ngerman, english]{babel}
 \usepackage[foot]{amsaddr}
\usepackage{amsfonts,amsmath}
\usepackage{amsthm}
\usepackage{txfonts}
\usepackage[arrow, matrix, curve]{xy}
\usepackage{tikz}

\newenvironment{example}[1][Example]{\begin{trivlist}
\item[\hskip \labelsep {\bfseries #1}]}{\end{trivlist}}
\newenvironment{remark}[1][Remark]{\begin{trivlist}
\item[\hskip \labelsep {\bfseries #1}]}{\end{trivlist}}

\newenvironment{notations}[1][Notations]{\begin{trivlist}
\item[\hskip \labelsep {\bfseries #1}]}{\end{trivlist}}

\newtheorem{theorem}{Theorem}[section]
\newtheorem{lemma}[theorem]{Lemma}
\newtheorem{proposition}[theorem]{Proposition}
\newtheorem{corollary}[theorem]{Corollary}
\newtheorem{definition}[theorem]{Definition}

\newcommand{\defi}{\coloneqq}

\newcommand{\R}{\mathbb{R}}
\newcommand{\Z}{\mathbb{Z}}
\newcommand{\N}{\mathbb{N}}
\newcommand{\Sph}{\mathbb{S}^1}

\newcommand{\Rzwei}{\mathbb{R}^2}
\newcommand{\fsig}{f}

\newcommand{\pt}{\partial_t}  
\newcommand{\ps}{\partial_s}    
\newcommand{\pp}{\partial_p}   
\newcommand{\ab}{[a,b]}
\newcommand{\nut}{[0,T)}

\newcommand{\ind}{ \mbox{ind}}

\newcommand{\de}{d}

\title[Singularities for a constraint curvature flow]{Singularities of the area preserving curve shortening flow with a free boundary condition}

\author[E. M\"ader-Baumdicker]{Elena M\"ader-Baumdicker}
\address{Karlsruhe Institute of Technology, Institute for Analysis, Englerstr. 2, 76131 Kalrsruhe, Germany}
\email{elena.maeder-baumdicker@kit.edu}

\begin{document}

\begin{abstract}
We consider the area preserving curve shortening flow with Neumann free boundary conditions outside of a convex domain or at a straight line. We give a criterion on initial curves that guarantees the appearance of a singularity in finite time. We prove that the singularity is of type II. Furthermore, if these initial curves are convex, then an appropriate rescaling at the finite maximal time of existence yields a grim reaper or half a grim reaper as limit flow. We construct examples of initial curves satisfying the mentioned criterion.
\end{abstract}

\maketitle

 \section{Introduction}
 The area preserving curve shortening flow (APCSF) for closed plane curves was introduced by M.\ Gage in 1986 \cite{Gage}. It is the ``steepest descent flow'' for the length functional under the constraint that the enclosed area is constant. For a family of simple closed curves $\gamma:\Sph\times[0,T)\to\Rzwei$, the evolution  equation turns out to be
 \begin{align*}
  \frac{d}{dt}\gamma=\left(\kappa-\frac{\int\kappa ds}{L}\right)\nu =\left(\kappa-\frac{2\pi}{L}\right)\nu,
 \end{align*}
where we use the following notation: $\nu=J\tau$ is the normal of the curves, where $J$ is the rotation by $+\frac{\pi}{2}$; $\kappa$ is the curvature with respect to $\nu$, $L$ is the length of the curves and $ds$ denotes integration by arclength.\\
M.\ Gage proved in \cite{Gage} that a strictly convex simple closed curve remains strictly convex under the APCSF. The curves converge for $t\to\infty$ smoothly to a circle enclosing the same enclosed area as $\gamma_0$. Thus, the flow converges to the solution of the isoperimetric problem in $\Rzwei$. This problem consists in finding the shortest closed curve enclosing a fixed area.
The analog result for $n$-surfaces in $\mathbb R^{n+1}$, $n\geq 2$ was proved by G.\ Huisken in \cite{Huisken87}: A uniformly convex, embedded surface moving according to the volume preserving mean curvature flow stays uniformly convex and exists for all times $t\in[0,\infty)$. The moving surfaces converge smoothly to a sphere enclosing the same volume as the initial surface.

We consider the APCSF in a free boundary setting and want to know when and how singularities develop. But at first we recall what is known about the existence of singularities in the closed situation. 

J.\ Escher and K.\ Ito considered in \cite{EscherIto} immersed closed curves possibly with self-intersections. Then the evolution equation is $ \frac{d}{dt}\gamma=\left(\kappa-\frac{2\pi m}{L}\right)\nu$ where $m\in\mathbb Z$ is the index (or turning number) of the immersed closed curves. The index $m$ is independent of time, and by possibly changing the orientation it is non-negative. Escher and Ito proved that an immersed curve with $m\geq 1$ and enclosed area $A_0<0$ or $m\geq 2$ and $L_0^2<4\pi m A_0$ develops a singularity in finite time. The proof is inspired by the work of K.-S.\ Chou on the surface diffusion flow for curves \cite{Chou}.\\

X.-L.\ Wang and L.-H.\ Kong also studied immersed closed curves moving according to the APCSF \cite{WangKong}. They proved that the flow exists for all times and converges smoothly to an $m$-fold circle when the initial curve is convex and has so-called ``$n$-fold rotational symmetry'' and index $m$ ($n>2m$). On the other hand, ``Abresch-Langer type'' curves either converge to a multiple cover of a circle (when $A_0>0$) or the curvature blows up at finite time (when $A_0<0$) or the curvature blows up at the maximal time of existence (when $A_0=0$), see \cite[Theorem~1.2]{WangKong}. 
Note that there are examples where only a slight change is necessary to deform an initial curve with $A_0<0$ into one with $A_0=0$ and then into one with $A_0>0$.
\\[-0.2cm]

We now explain the free boundary setting of the APCSF which was studied by the author in \cite{MeinPaper, MeineDiss}. Let $\Sigma\subset\Rzwei$ be a convex simple closed curve in the plane and orient it positively. We call $\Sigma$ a \emph{support curve}. It is not moving in time. An \emph{initial curve} $\gamma_0:[a,b]\to\Rzwei$ is a curve with endpoints $\gamma_0(a),\gamma_0(b)\in\Sigma$ where we prescribe the angle to be $90$ degrees. We consider the ``outer situation'' which means that the curve $\gamma_0$ goes into the ``exterior domain'' with respect to $\Sigma$ and also comes back to $\Sigma$ ``from the outside'' at the endpoints. In formulas, this means
\begin{align}\label{perp}
 \tau_0(a) = - \nu_\Sigma (\gamma_0(a)), \ \ \ \ \tau_0(b) = \nu_\Sigma (\gamma_0(b)),
\end{align}
where $\tau_0:[a,b]\to\R^2$ is the tangent of $\gamma_0$ and $\nu_\Sigma:\Sigma \subset \R^2\to\R^2$ is the inner unit normal to $\Sigma$\footnote{As $\Sigma$ is a simple closed curve, we define the unit normal (and the tangent) to be defined on the image of the curve in $\R^2$. Since $\gamma_0$ can have self-intersections, we use the parametrized version of the tangent.}.

We now let the curve $\gamma_0$ flow according to the APCSF such that these conditions are preserved, i.e.\ $\gamma:[a,b]\times\nut\to\Rzwei$ satisfies $\gamma(a,t),\gamma(b,t)\in\Sigma$ and (\ref{perp}) for each time $t\in\nut$ and $$\frac{d}{dt}\gamma=\left(\kappa-\frac{\int\kappa ds}{L}\right)\nu.$$ 
As the curves are not closed the quantity $\int\kappa ds$ is not an integer times $2\pi$ in general. It is in fact the first step to find conditions that guarantee a bound of $\bar\kappa:= \frac{\int\kappa ds}{L}$ independent of $t$. In \cite{MeinPaper}, the author proved that the flow in this setting does not develop a singularity when the initial curve satisfies four conditions: 
\begin{enumerate}
 \item  $\gamma_0$ is strictly convex,
 \item it is embedded,
 \item it is contained in the exterior domain with respect to $\Sigma$ and 
 \item  it satisfies $L_0< \frac{4}{5\max|\kappa_\Sigma|}\arcsin\left(\frac{A_0}{L_0^2}\right)$, 
\end{enumerate}
where $A_0$ is the enclosed area of the domain enclosed by $\gamma_0$ and the part of $\Sigma$ connecting $\gamma_0(b)$ and $\gamma_0(a)$. Furthermore, the curves $\gamma(\cdot,t)$ subconverge under these conditions smoothly for $t\to\infty$ to an arc of a circle sitting outside of $\Sigma$ and meeting $\Sigma$ perpendicularly. \\
 
In this paper we answer the following questions that naturally arise when studying this setting:
\begin{itemize}
 \item Are there curves that develop a singularity under the APCSF in the free boundary setting?
 \item Are there convex initial curves developing a singularity?
 \item Does the singularity appear in finite time?
 \item Of what type are the singularities? 
 \item What does a blowup at the singular time look like?
\end{itemize}
For our main theorem we explain some preliminaries. As $\Sigma$ is a smooth convex closed curve, every $x\in\Sigma$ has an ``antipodal point'' $x'\in\Sigma$ which is a point in $\Sigma$ with $\tau_\Sigma(x)=-\tau_\Sigma(x')$, where $\tau_\Sigma:\Sigma\subset \R^2\to\R^2$ is the tangent of $\Sigma$. Note that this point is not unique as the curve is not strictly convex. The \emph{minimum width} of $\Sigma$ is 
\begin{align*}
 d_\Sigma \defi \min\{|x-x'|:x, x'\in\Sigma, x' \text{ antipodal to } x\}.
\end{align*}
This is the least distance of two parallel lines touching $\Sigma$.\\[-0.2cm]

We consider $\gamma_0:\ab\to\R^2$, an initial curve with $L_0< d_\Sigma$, where $L_0$ is the length of $\gamma_0$. By definition of $d_\Sigma$ the points $\gamma_0(a)$ and $\gamma_0(b)$ can not be antipodal to each other. We let the curve $\gamma_0$ flow by the APCSF with Neumann free boundary conditions as described above. As this flow is the ``steepest descent flow'' of the length functional (under a constraint), the length does not increase under the flow. As a consequence we get that all endpoints of the evolving curves $\gamma(a,t),\gamma(b,t)$ are not antipodal to each other. Note that for each time $t\in[0,T)$ the curve $\Sigma\setminus\{\gamma(a,t),\gamma(b,t)\}$ is divided into two pieces. At one piece the angle of the normal $\nu_\Sigma$ turns more than $\pi$. The angle of the unit normal of the other part, we call it the \emph{short piece}, turns less than $\pi$. \\[-0.2cm]

For each $t\in[0,T)$ we append the ``short piece'' of $\Sigma$ to $\gamma(\cdot,t)$ in order to close the curve $\gamma(\cdot,t)$: Define a family $\sigma(t):[\alpha(t),\beta(t)]\to\Sigma$ by connecting $\gamma(b,t)$ and $\gamma(a,t)$ by following $\Sigma$ along the ``short piece''. Note that $\sigma(t)$ is just a point if $\gamma(a,t)=\gamma(b,t)$. 
We use the notation $\sigma(0)=:\sigma_0$. Since the endpoints of our curves are never antipodal and as the endpoints of $\gamma(\cdot,t)$ vary continuously in $t$, the family $\sigma$ is continuous in $t$. We will see that it is actually $C^1$ in $t$.
We denote the assembled closed curve by $\gamma(\cdot,t) + \sigma(t)$.  The boundary conditions imply that $\int_{\gamma(\cdot,t)}\kappa d s\not \in 2\pi \Z$ for all $t\in[0,T)$, in particular $\int_{\gamma_0}\kappa d s\neq 0$. 
The (oriented) enclosed area $A(\gamma(\cdot,t) + \sigma(t))$ is preserved under the APCSF, and we can state our main theorem:

\begin{theorem}\label{thm12}
 Let $\gamma_0:\ab\to\Rzwei$ be an initial curve satisfying $L_0<d_\Sigma$. Choose the orientation of $\gamma_0$ such that $\int_{\gamma_0}\kappa d s>0$. Fix $l\in\N$ such that $(2l-2)\pi<\int_{\gamma_0}\kappa ds<2l\pi$. We further assume
 \begin{enumerate}
  \item either $A(\gamma_0 +\sigma_0)>0$ and $\frac{L_0^2}{A(\gamma_0 +\sigma_0)} \leq \pi\frac{(2l-1)^2}{l}$,
  \item or 
    $A(\gamma_0 +\sigma_0)<0$,
   \end{enumerate}
   where $\gamma_0 + \sigma_0$ is the extension of $\gamma_0$ along the ``short piece'' described above.\\
 In these cases  the solution of the area preserving curve shortening flow with Neumann free boundary conditions outside of $\Sigma$ develops a singularity in finite time, i.e.\ $T_{\text{max}}<\infty$.\\
 Furthermore, the finite time singularity is of type II in the sense that 
 \begin{align*}
 & \max_{p\in\ab}|\kappa|(p,t)\to\infty \ (t\to T_{\text{max}}) \text{  and }\\
 & \max_{p\in\ab}\left(|\kappa|^2(p,t)(T_{\text{max}}-t)\right) \text{ is unbounded}.
 \end{align*}
\end{theorem}

If $\gamma_0$ is convex, we can say what the limit flow looks like after a suitable rescaling procedure.
\begin{corollary} \label{corconvex}
Let $\gamma_0:[a,b]\to\R^2$ be an initial curve satisfying the conditions from Theorem~\ref{thm12}. Assume further that $\gamma_0$ is convex, $\kappa_0\geq0$.
Then the ``Hamilton blow-up'' at $T_{max}<\infty$ yields either a grim reaper without boundary or half a grim reaper at a straight line. 
\end{corollary}

\begin{remark}
\begin{enumerate}
 \item The ``Hamilton blow-up'' was defined in \cite{Hamilton2}. We will explain it in the proof of Corollary~\ref{corconvex}.
 \item There is numerical evidence given by U.\ F.\ Mayer \cite{Mayer} that there are embedded closed curves that at first get a self-intersection and then develop a singularity under the APCSF. In the free boundary setting, it seems to be the case that there are initially embedded curves that stay embedded but develop a singularity in finite time, see Example Three in Section~\ref{sec3}. We think that these curves develop a singularity at the boundary.
\end{enumerate}
\end{remark}
We also study the situation at a straight line. The result is as follows.
\begin{theorem}\label{prop}
Let $\gamma_0:\ab\to\R^2$ be an initial curve at a straight line $\Sigma$. Let $\delta_0$ be the closed curve obtained by reflecting $\gamma_0$ at $\Sigma$. Let $\ind(\delta_0)=:m$ be the index of $\delta_0$. Then $m$ is odd. Choose the orientation of $\delta_0$ such that $m$ is positive.\\
Then the area preserving curve shortening flow with Neumann free boundary conditions at the line $\Sigma$ develops a singularity in finite time if one of the following conditions is satisfied:
  \begin{enumerate}
   \item Either $A(\delta_0)<0$.
   \item Or $m\geq 3$ and $L(\delta_0)^2 < 4\pi m A(\delta_0)$.
\end{enumerate}
The singularity is of type II.
\end{theorem}
The structure of this paper is as follows. In Section~\ref{sec2} we recall some results from \cite{MeineDiss, MeinPaper} that we use in the proof of Theorem~\ref{thm12}. We explain again how strongly the condition $L_0<d_\Sigma$ influences the the behavior of $\int_{\gamma(\cdot,t)}\kappa d s$ along the flow. A bound on $|\bar\kappa|$ independent of $T_{max}$ is a consequence. If $T_{max}=\infty$, then the bound on $|\bar\kappa|$ together with \cite{MeinPaper} imply
subconvergence to a part of a circle that is possibly (partly) multicovered. We study the geometry of the limiting arc and get a contradiction to the assumptions. We refine results from \cite{MeinPaper} to show that the singularity is of type II. If the initial curve is convex we showed in \cite{MeinPaper} that the ``Hamilton blowup'' yields a grim reaper or half a grim reaper at a straight line.\\ 

In Section~\ref{sec3}, we give examples of curves that do satisfy the conditions of Theorem~\ref{thm12} and Corollary~\ref{corconvex}. \\

Section~\ref{sec4} contains the proof of Theorem~\ref{prop}. We reflect the curves at the line $\Sigma$ and apply the results from \cite{EscherIto}. We combine this with results from \cite{MeinPaper} to show that the singularity is of type II.

\section*{Acknowledgment}
 The author would like to thank Jonas Hirsch for very useful discussions. Furthermore, the author would like to express her gratitude to the referee for all the useful comments and suggestions. The author is funded by the Deutsche Forschungsgemeinschaft (DFG), LA 3444/1-1 and MA 7559/1-1. 
 \section{Singularities of type II in finite time} \label{sec2}
 
 \begin{notations}
 Let $\gamma:\ab\to\Rzwei$ be a piecewise smooth, regular curve and let $h:[a,b]\to\R^n$, $n\in\{1,2\}$, be a $C^1$-map, $h=h(p)$. 
  We denote by $\ps h\defi \frac{1}{|\pp \gamma|}\pp h$ the derivative with respect to arclength of $h$. We define $d s \defi |\partial_p \gamma|d p$. We recall the formula for the curvature of $\gamma$
  \begin{align*}
 \kappa (p) = \langle \partial^2_s \gamma(p),\nu(p)\rangle,
\end{align*}
where $\nu= J\tau=J\ps \gamma$ is the normal of the curve $\gamma$, $
 J$ is the rotation by $+\frac{\pi}{2}$ in the plane.
 \end{notations}

 \begin{definition}
  We call a smooth, regular, convex, simple and smoothly closed curve $ f:\mathbb S^1\to\Rzwei$ a \emph{support curve}. We assume $\fsig$ to be parametrized by arclength.
  We orient $f$ positively so that $ \kappa_\Sigma\geq 0$. We use the notation $$\Sigma\defi  f\left(\mathbb S^1\right). $$
  The curve $\Sigma$ separates $\Rzwei$ into a bounded and an unbounded domain. The bounded domain is enclosed by $\Sigma$ and is denoted by $G_\Sigma$.\\
  We define $d_\Sigma \defi \min\{|x-y|:x,y\in\Sigma, \tau_\Sigma(x)=-\tau_\Sigma(y)\}$, the smallest distance between two parallel lines in $\Rzwei$ that touch $G_\Sigma$ (the \emph{minimum width}). 
\end{definition}

\begin{definition}
 A planar, smooth, regular curve $\gamma_0:[a,b]\to\Rzwei$ is called \emph{initial curve} if it satisfies the conditions
\begin{align*}
 \gamma_0(a),\gamma_0(b)&\in\Sigma\\
 \tau_0(a) =  -\nu_\Sigma (\gamma_0(a)), &\ \ \ \ \tau_0(b) = \nu_\Sigma (\gamma_0(b)),
\end{align*}
  where $\tau_0=\ps \gamma_0$ is the tangent of $\gamma_0$ and $\nu_\Sigma = J\, \ps f \circ f^{-1}:\Sigma \to\R^2$ is the inner unit normal of~$\Sigma$ (defined on the image $\Sigma= f\left(\mathbb S^1\right)$). 
\end{definition}

\begin{definition}
 Let $\gamma_0:[a,b]\to\Rzwei$ be an initial curve. A smooth family of smooth, regular curves $\gamma:[a,b]\times[0,T)\to \Rzwei $ that satisfies  
\begin{alignat}{2}
 \frac{\partial \gamma}{\partial t} (p,t) = ( \kappa(p,t)& - \bar \kappa(t))\nu(p,t)\hspace*{1.5cm}   & \forall (p,t)&\in [a,b]\times [0,T),\nonumber\\
  \gamma(p,0)& =\gamma_0(p) & \forall p&\in [a,b],\nonumber \\[0.2cm]
    \gamma(a,t),\gamma(b,t)&\in\Sigma &\forall t& \in [0,T),\label{1}\\[0.1cm]
 \tau(a,t) = - \nu_\Sigma (\gamma(a,t)), &\ \  \tau(b,t) = \nu_\Sigma (\gamma(b,t)), & \forall t &\in[0,T),\nonumber
\end{alignat}
is called a \emph{solution of the area preserving curve shortening problem with Neumann free boundary conditions}. Here, $\bar\kappa$ denotes the average of the curvature,
\begin{align*}
 \bar\kappa(t)\defi \frac{\int \kappa(p,t)d s}{\int d s}=  \frac{\int \kappa(p,t)d s}{L(\gamma(\cdot,t))},
\end{align*}
and $\nu_\Sigma$ is the \emph{inner} unit normal of $\Sigma$. 
Here and in the rest of the paper, we use the notation $\gamma_t\defi \gamma(\cdot,t)$.

\end{definition}

\begin{remark}
 For a smooth initial curve, existence and uniqueness of the solution of (\ref{1}) is standard. One gets short time existence on a short time interval $[0,T_0]$. The solution can be extended up to a maximal time of existence $T_{max}\leq\infty$. By regularity theory for parabolic Neumann problems the curves satisfy 
 \begin{align*}
 \gamma \in C^{2+\alpha, 1 +\frac{\alpha}{2}}\left([a,b]\times [0,T_{max}),\mathbb R^2\right)\cap C^\infty\left([a,b]\times (0,T_{max}),\mathbb R^2\right) ,\ \alpha\in(0,1),
\end{align*}
where $C^{2+\alpha, 1 +\frac{\alpha}{2}}$ denotes the usual parabolic H\"older space.
 If $T_{max}<\infty$ then $\max_{\ab}|\kappa|(\cdot,t)\to\infty$ ($t\to T_{max}$). A source for the existence for closed curves moving by a geometric flow with a constraint is for example \cite{DziukKuwertSch}. The technique how to transform the free boundary problem into a standard Neumann boundary problem can be found in \cite{StahlPaper1, StahlPaper2}. For our specific situation a sketch of the existence and regularity result is in \cite[Proposition~2.4]{MeinPaper}.
\end{remark}

    \begin{definition}
     Let $\Sigma$ be a support curve and let $\gamma:[a,b]\to\Rzwei$ be a curve with $\gamma(a),\gamma(b)\in\Sigma$. Then we call a curve $\sigma:[\tilde a,\tilde b]\to\Sigma\subset\Rzwei$ with $\sigma(\tilde a)=\gamma(b)$ and $\sigma(\tilde b)=\gamma(a)$ a \emph{boundary curve on $\Sigma$ with respect to $\gamma$}.
    \end{definition}
 
    \begin{definition} \label{orientedVolume} 
      Let $\gamma:\ab\times[0,T)\to\Rzwei$ be a solution of (\ref{1}).
      Consider  a $C^1$-family of smooth curves $\sigma:[\tilde a,\tilde b]\times [0,T)\to\Sigma$ with $\sigma(\tilde a,t)=\gamma(b,t)$ and $\sigma(\tilde b,t)=\gamma(a,t)$ for all $t\in[0,T)$, i.e.~$\sigma_t\defi\sigma(\cdot,t)$ is a boundary curve on $\Sigma$ with respect to $\gamma_t = \gamma(\cdot,t)$. 
      Then for each $ t\in[0,T)$, we call the following expression the \emph{oriented area enclosed by $\gamma_t$ and $\Sigma$:}
      \begin{align} \label{enclosedVolume}
       A(\gamma_t + \sigma_t)\defi \frac{1}{2} \int_{\gamma_t} p^1 \de p^2 - p^2 \de p^1 + \frac{1}{2} \int_{\sigma_t} p^1 \de p^2 - p^2 \de p^1.
      \end{align}
    \end{definition}
    \begin{remark}
     Our curves $\gamma_t$ are regular. But it can happen that a curve $\sigma_t$ is not regular. For our situation, this will only happen if $\gamma_t(a)=\gamma_t(b)$. Then $\sigma_t$ will be just the point $\sigma_t \equiv \gamma_t(a)=\gamma_t(b)$. This is not important for the definition of the enclosed area because in such a situation $\gamma_t$ is already closed and the second integral in (\ref{enclosedVolume}) vanishes.
    \end{remark}

    We recall some basic properties proved in \cite{MeinPaper}.
    \begin{lemma}[Lemma~2.6, Lemma~2.11 and Corollary~2.14 \cite{MeinPaper}] \label{lemma:areapres}
    Let $\gamma_0:\ab\to\R^2$ be a smooth initial curve.
     Then we have the following properties: The area preserving curve shortening flow is curve shortening and area preserving, i.e.\ $\frac{d}{dt}L(\gamma_t)\leq 0$ and $\frac{d}{dt}A(\gamma_t,\sigma_t)=0$ on $\nut$, where $\gamma:\ab\times\nut\to\R^2$ is a solution of (\ref{1}) and $\sigma:[\tilde a,\tilde b]\times\nut\to\Sigma$ is a $C^1$-family of boundary curve on $\Sigma$ with respect to $\gamma$.\\
     As the domain $G_\Sigma$ is convex and as $\gamma_0$ goes into $\R^2\setminus G_\Sigma$ at $\gamma_0(a)$ and comes back to $\Sigma$ from $\R^2\setminus G_\Sigma$ at $\gamma_0(b)$ the flow improves convexity to strict convexity\footnote{The author emphasizes that convexity is probably not preserved if one allows the curve to meet $\Sigma$ perpendicularly from \emph{inside} $G_\Sigma$ at the endpoints.}. This is, $\kappa_0\geq 0$ for the initial curve implies $\kappa>0$ on $\ab\times(0,T)$.
    \end{lemma}
    \begin{remark}\begin{enumerate}
                   \item  If $\gamma_0$ is a smooth initial curve then a $C^1$-family of boundary curves $\sigma$ on $\Sigma$ with respect to $\gamma$ exists. This was proved in \cite[Lemma~2.9]{MeinPaper}. Under the condition $L_0<d_\Sigma$ we will explain the construction of such a family below.
                   \item We emphasize that it is allowed that one of the boundary curves $\sigma_t$ consists only of one point (namely of the endpoints $\gamma_t(a)=\gamma_t(b)$). Important in the proof of Lemma~\ref{lemma:areapres} is only that one has to find a family of boundary curves where the enclosed area is \emph{continuous} in $t$.
                  \end{enumerate}
    \end{remark}
    
    \begin{lemma}[Construction of the boundary curves]\label{constr}
    Let $\gamma_0:[a,b]\to\R^2$ be a smooth initial curve with $L_0< d_\Sigma$. Then the solution of (\ref{1}) has the following property: \\
    The endpoints $\gamma_t(a), \gamma_t(b)$ divide $\Sigma$ into two pieces for each $t\in[0,T)$. The angle of the unit normal of one component of $\Sigma\setminus\{\gamma_t(a),\gamma_t(b)\}$ turns more than $\pi$ (and less or equal than $2\pi$). The unit normal of other component -- we will call it \emph{the short piece} -- turns an angle of less than $\pi$. Note that the (degenerate) case where the short piece is just a point is possible. This only happens if $\gamma_t(a)=\gamma_t(b)$.\\
    We denote by $\sigma(t)$ the curve from $\gamma_t(b)$ to $\gamma_t(a)$ along the short piece of $\Sigma$. After reparametrizations we get a $C^1$-family of boundary curves $\sigma:[\tilde a,\tilde b]\times [0,T)\to\Sigma$ with respect to $\gamma$, where $\sigma_t$ are regular smooth curves except in the degenerate case where $\sigma_t\equiv\gamma_t(a)=\gamma_t(b)$. \\
    As a consequence, the enclosed area $A(\gamma_t + \sigma_t)$ is constant along the flow.
    \end{lemma}
    \begin{remark}
     The ``short piece'' is not the piece with the shorter length. It is the piece where the image of the unit normal on $\mathbb S^1$ is shorter.
    \end{remark}
    
    \begin{proof}
     The construction of the boundary curve is quite explicit. The only thing that we have to show is that $\sigma_t$ is $C^1$ (and in particular continuous) with respect to $t$. The continuity follows from that fact that $L(\gamma_t)\leq L_0< d_\Sigma$. By this property the short piece cannot jump from time to time, i.e.\ the short piece of $\Sigma$ varies continuously in $t$. Since $\gamma$ is in fact $C^1$ in $t$ and as $\Sigma$ is smooth, $\sigma$ is a $C^1$ family of boundary curves.
    \end{proof}

The following result comes from analyzing the geometric properties of a convex curve that satisfy the Neumann free boundary conditions outside a convex domain at the endpoints.

\begin{proposition}\label{wi}
Let $\Sigma\subset\R^2$ be a positively oriented convex smooth Jordan curve and let $\gamma:\ab\to\R^2$ be a $C^2$-curve with $\kappa>0$ and
\begin{align*}
 \gamma(a),\gamma(b)&\in\Sigma,\\
\tau(a) = - \nu_\Sigma (\gamma(a)), &\ \  \tau(b) = \nu_\Sigma (\gamma(b)),
\end{align*}
where $\nu_\Sigma$ is the inner unit normal of $\Sigma$. 
 Then we have that $\int\kappa d s\geq \pi$. 
\end{proposition}

\begin{proof}
In \cite[Proposition~3.1]{MeinPaper}, it was shown that the geometric situation of the curves imply $\int\kappa d s\geq \pi$. The statement there was formulated for a solution of (\ref{1}). But the only properties of the curves that are used in the proof are strict convexity and the boundary conditions. 
\end{proof}

In order to be able to use results from \cite{MeinPaper} we need to show that $\bar\kappa(t)$ is bounded in $L^\infty$. As we want to show results about flows with infinite lifespan, we want the bound to be independent of the maximal time of existence $T_{max}$.

\begin{proposition}\label{est:barkappa}
 Let $\gamma_0:\ab\to\Rzwei$ be an initial curve (not necessarily convex) with $L_0<d_\Sigma$. Consider the solution of the APCSF (\ref{1}) on the maximal time interval of existence $[0,T_{max})$. Choose $l\in \Z$ such that $(2l - 2)\pi<\int_{\gamma_0}\kappa d s<2l\pi$. Then we have that
 \begin{align*}
  (2l-2) \pi < \int_{\gamma_t} \kappa  d s < 2l\pi \ \ \text{ for all } t\in[0,T_{max}).
 \end{align*}

\end{proposition}
    
     \begin{proof}
      By definition of $ d_\Sigma$ and by the curve shortening property the points $\gamma_t(a)$ and $\gamma_t(b)$ are never ``antipodal points''. This means that $\tau_\Sigma(\gamma_t(a))\not = -\tau_\Sigma(\gamma_t(b))$ for each $t\in[0,T)$. Taking into account the boundary conditions $\nu_\Sigma(\gamma_t(a)) = -\tau(a,t)$ and $\nu_\Sigma(\gamma_t(b)) =\tau(b,t)$ for the inner unit normal $\nu_\Sigma= J\tau_\Sigma$ we get that $$\tau(a,t)\not =\tau(b,t)$$ for each $t\in[0,T)$. This particularly implies that $\int_{\gamma_t}\kappa d s\not \in 2\pi\Z$ for each $t\in[0,T)$. The continuity of $\int_{\gamma_t}\kappa d s$ with respect to $t$ implies the result.
     \end{proof}

\begin{proposition}\label{prop:Lgr0}
Let $\gamma:\ab\times[0,T_{\text{max}})\to\Rzwei$ be the solution of (\ref{1}) where the initial curve $\gamma_0:\ab\to\Rzwei$ satisfies $L_0< d_\Sigma$. 
Furthermore, we assume that $\gamma_0$ satisfies
\begin{align*}
 A(\gamma_0 +\sigma_0)\neq 0,
\end{align*}
where $\gamma_0 + \sigma_0$ is the extension of $\gamma_0$ via the ``short piece'' along $\Sigma$ defined in Lemma~\ref{constr}. 
Then there is a constant $\delta>0$ such that $L(\gamma_t)\geq \delta$ for all $t\in[0,T_{\text{max}})$.

\end{proposition}

\begin{proof}
 We assume that there is a sequence $t_j\to T_{\text{max}}$  with $L(\gamma_{t_j})\to 0$ $(j\to\infty)$. 
 Since $\Sigma$ is compact we get $x_0\in\Sigma$ and (after passing to a subsequence) $\gamma(a,t_j)\to x_0$, $\gamma(b,t_j)\to x_0$. This means that the curves $\gamma_{t_j}$ close up as $j\to\infty$. The boundary curves $\sigma(t_j)$ are the curves connecting the endpoints $\gamma_{t_j}(b)$ and $\gamma_{t_j}(a)$ along the part of $\Sigma$ where $\int_{\sigma_{t_j}}\kappa_\Sigma d s_\Sigma$ is smaller. This implies that $L(\sigma_{t_j})\to 0$ as $j\to\infty$. As a consequence, we also have that $A(\gamma_{t_j} + \sigma_{t_j})\to 0$ as $j\to\infty$.\\
 Due to the fact that $A(\gamma_0 + \sigma_0)=A(\gamma_{t_j} + \sigma_{t_j})$ for all $j\in\N$ we get a contradiction to our assumption.

\end{proof}

\begin{theorem}\label{convthm}
 Let $\gamma:\ab\times[0,\infty)\to\Rzwei$ be a solution of (\ref{1}) (without singularities in finite time) where the initial curve $\gamma_0:\ab\to\Rzwei$ satisfies  $L_0< d_\Sigma$ and
\begin{align*}
  A(\gamma_0 +\sigma_0)\neq 0.
\end{align*}
Here, $\gamma_0 + \sigma_0$ is the extension of $\gamma_0$ along the ``short piece'' of $\Sigma$ coming from Lemma~\ref{constr}. Choose $l\in \Z$ such that $(2l - 2)\pi<\int_{\gamma_0}\kappa d s<2l\pi$.\\
Then $\gamma_t$ $(t\to\infty)$ subconverges (after reparametrization) smoothly to a (possibly multicovered) arc of circle $\gamma_\infty$ sitting outside of $\Sigma$ at the endpoints. Note that the arc can be positively or negatively oriented. Each of the two contact angles at the endpoints of $\gamma_\infty$ is a $90$ degrees angle. Furthermore, the limit curve satisfies
\begin{align}
 \int\kappa d s_\infty&\in \left[(2l-1)\pi,2l\pi\right) \text{ if } l \geq 1, \label{posi}\\
 \int\kappa d s_\infty&\in \left((2l-2)\pi,(2l-1)\pi\right] \text{ if } l\leq 0. \label{nega}
\end{align}
\end{theorem}

\begin{proof}
In \cite[Theorem~7.15]{MeinPaper}, subconvergence is proved under the conditions $L(\gamma_t)\geq c_1>0$ and $\bar\kappa(t)\in [\bar c, c_2]$ for all $t\in[0,\infty)$ for constants $c_1, \bar c, c_2>0$. But the proof in fact also works if we do not assume the lower bound $\bar\kappa\geq \bar c$. We only need $|\bar\kappa|\leq c_2$ and $L(\gamma_t)\geq c_1>0$. We sketch this proof for the convenience of the reader:\\
For any sequence $\tau_l\to\infty$ we reparametrize the original curves $\tilde\gamma(\cdot,\tau_l)$ by constant speed and get a solution $\gamma_l:[0,1]\times[0,\infty)\to\R^2$ of (\ref{1}) with $|\gamma_l'|=L(\tilde\gamma_{\tau_l})$ at the time $\tau_l$. Using Gagliardo-Nirenberg interpolation inequalities and integral estimates we proved in Corollary~7.14 from \cite{MeinPaper} a bound 
\begin{align*}
 \sup_{(p,t)\in[0,1]\times[1,\infty)}|\kappa_l(p,t)|\leq C,
\end{align*}
where $C$ does not depend on $l$. Using the graph representation of the curves, the lower bound on the length and the flow equation we get estimates $|\pt^i\ps^m\kappa|\leq c$ on $[0,1]\times [\tau_l,\tau_l + \delta]$ for any $\delta>0$. We split the derivatives $\pp\gamma_l$ into its tangential and normal part and use an induction argument together with the bound on $|\ps^m\kappa|$. This yields $|\pp^m\gamma_l|\leq c$ on $[0,1]\times[\tau_l,\tau_l+ \delta]$, where $c$ depends on $m,\Sigma,C,L_0$ and $\delta$. Choose $\tau_l\to\infty$ and $\delta>0$ such that $\bigcup_{l\in\N}[\tau_l,\tau_l + \delta)=[1,\infty)$ then we have proved
\begin{align*}
 |\pp^m\gamma_l|\leq c \text{ on }[0,1]\times[1,\infty).
\end{align*}
The proof of these estimates can be found in \cite[Proof of Proposition~4.7]{MeinPaper} or in \cite[Section~5.3]{MeineDiss}.\\
For any $t_l\to\infty$ we consider $\alpha_l:=\gamma_l(\cdot,t_l)$. Using the theorem of Arzela-Ascoli the curves subconverge to a smooth curve $\gamma_\infty:[0,1]\to\R^2$ in every $C^m$ on $[0,1]$, $m\in\N_0$. This implies 
\begin{align*}
 \lim_{l'\to\infty}\bar\kappa(t_{l'})= \lim_{l'\to\infty}\frac{\int_{\alpha_{l'}}\kappa ds}{\int_{\alpha_{l'}}  ds} =\bar\kappa(\gamma_\infty)\in [-c_2,c_2].
\end{align*}
As a consequence we get that
\begin{align*}
 \lim_{l'\to\infty}\int_{\alpha_{l'}}(\kappa-\bar\kappa(\gamma_\infty))^2 d s = \lim_{l'\to\infty}\int_{\gamma(\cdot,t_{l'})}(\kappa-\bar\kappa)^2 d s =0,
\end{align*}
where we used $\lim_{t\to\infty}\int_{\gamma_t}(\kappa-\bar\kappa)^2 ds = 0$, which was shown in Corollary~7.5 in~\cite{MeinPaper}.  Thus, the limit curve $\gamma_\infty$ satisfies $\kappa_\infty \equiv \bar\kappa(\gamma_\infty)\in[-c_2,c_2]$.
By compactness of $\Sigma$ and by continuity we get that the endpoints of $\gamma_\infty$ lie in $\Sigma$, the curve goes into the ``exterior'' domain and comes back from the ``exterior'' domain at the endpoints. Is is not possible that $\gamma_\infty$ is a part of a straight line by these geometric properties, which implies that $\bar\kappa(\gamma_\infty)\neq 0$.\\
So we get that the limit curve $\gamma_\infty$ is a (possibly partly multicovered) arc of a circle. By reversing the orientation we can assume that $\gamma_\infty$ is positively oriented, thus $\kappa_\infty\equiv\bar\kappa(\gamma_\infty)>0$. Proposition~\ref{est:barkappa} yields 
\begin{align*}
  \int\kappa d s_\infty\in \left[(2l-2)\pi,2l\pi\right].
\end{align*}
We showed in Proposition~\ref{wi} that for a strictly convex curve ``outside'' of $\Sigma$ at the endpoints we always have $\int\kappa ds \geq \pi$. Using this for the ``last'' open part of the arc $\gamma_\infty$ we get that $\int\kappa d s_\infty\in \left[(2l-1)\pi,2l\pi\right]$. The situation $\int\kappa d s_\infty =2\pi l$ is excluded by the geometric situation as well.\\
If the arc was negatively oriented, estimate (\ref{nega}) is obtained by using (\ref{posi}) for the limiting arc with reversed orientation. \\

It remains to mention that the bounds $L(\gamma_t)\geq c_1>0$ and $|\bar\kappa|\leq c_2$ are satisfied under the assumptions of the theorem. This follows from Proposition~\ref{prop:Lgr0} and Proposition~\ref{est:barkappa}.

\end{proof}

We restate our result about the existence of finite time singularities.

\begin{theorem}\label{singu}
 Let $\gamma_0:\ab\to\Rzwei$ be an initial curve with $L_0< d_\Sigma$. Choose the orientation of $\gamma_0$ such that $\int_{\gamma_0}\kappa ds >0$. Consider $l\in\N$ such that $\int_{\gamma_0}\kappa ds \in \left((2l-2)\pi,2l\pi\right)$. We further assume
 \begin{enumerate}
  \item either $A(\gamma_0 + \sigma_0)>0$ and $\frac{L_0^2}{A(\gamma_0 + \sigma_0)} \leq \pi \frac{(2l-1)^2}{l}$,
  \label{case11}
  \item or $A(\gamma_0 + \sigma_0)<0$. \label{case21}
 \end{enumerate}
 Here, $\gamma_0 + \sigma_0$ is the extension of $\gamma_0$ along the ``short piece'' of $\Sigma$ defined in Lemma~\ref{constr}. \\ 
 In both cases  the solution of (\ref{1}) develops a singularity in finite time, i.e.\ $T_{\text{max}}<\infty$. 

\end{theorem}

\begin{proof}
 Theorem~\ref{convthm} implies that $\gamma_t$ subconverges to an arc of a circle $\gamma_\infty$ sitting outside of $\Sigma$ at the endpoints. 
 Property $l>0$ implies (\ref{posi}), which is $\int\kappa d s_\infty\in \left[(2l-1)\pi,2l\pi\right)$. This also gives us the information that the arc $\gamma_\infty$ is positively oriented.\\
In particular, the enclosed area in the limit is positive, $A(\gamma_\infty + \sigma_\infty)>0$, which yields a contradiction in Case~\ref{case21}) because the flow is area preserving.\\
We consider Case~\ref{case11}): The quantities in the isoperimetric quotient satisfy
\begin{align}\label{gds}
 L(\gamma_\infty)& = 2(l-1)\pi r_\infty + \alpha_\infty r_\infty \geq (2l-1)\pi r_\infty\ \text{ for some }\alpha_\infty\in[\pi,2\pi)\text{ and }\\
A(\gamma_0 +\sigma_0) &= A(\gamma_\infty + \sigma_\infty)= (l-1) \pi r_\infty^2 + \tilde A_\infty < l\pi r_\infty^2, \label{bds}
\end{align}
 where $r_\infty$ is the radius of the arc $\gamma_\infty$ and $0<\tilde A_\infty< \pi r_\infty^2$ is the area of the domain inside one full circulation of $\gamma_\infty$ without the positive area of $G_\Sigma$. We compute
\begin{align} 
 \frac{L_0^2}{A(\gamma_0 + \sigma_0)}&\geq \frac{L(\gamma_{t_j})^2}{A(\gamma_0 +\sigma_0) } \to \frac{L(\gamma_\infty)^2}{A(\gamma_0 +\sigma_0) } \text{ as } t_j\to\infty. \label{gfs}
\end{align}
We use (\ref{gds}) and (\ref{bds}) and the fact that the enclosed area is preserved and get
\begin{align*}
 \frac{L_0^2}{A(\gamma_0 + \sigma_0)} > \frac{(2l-1)^2 \pi^2 r_\infty^2}{l \pi r_\infty^2}= \pi\frac{(2l-1)^2}{l},
\end{align*}
which contradicts our assumptions.

\end{proof}

\begin{remark}
 The result of the previous theorem can be improved by analyzing the geometric situation in the limit more carefully. Instead of using the estimate $\tilde A_\infty<\pi r_\infty^2$  we can prove $\tilde A_\infty< \pi r_\infty^2 (1-\frac{7}{20\pi})$. If $A(\gamma_0 + \sigma_0)>0$ we get that $\frac{L_0^2}{A(\gamma_0 +\sigma_0)}<\pi\frac{(2l-1)^2}{l-\frac{7}{20\pi}}$ implies a singularity in finite time. 
 This is again not sharp because we estimated some geometric constants.
\end{remark}

\begin{corollary} \label{typeII}
 Let $\gamma_0:\ab\to\Rzwei$ be an initial curve satisfying the conditions from Theorem~\ref{singu}. Then the finite time singularity is of type II in the sense that 
 \begin{align*}
 & \max_{p\in\ab}|\kappa|(p,t)\to\infty \ (t\to T_{\text{max}}) \text{  and }\\
 & \max_{p\in\ab}\left(|\kappa|^2(p,t)(T_{\text{max}}-t)\right) \text{ is unbounded}.
 \end{align*}
\end{corollary}

The proof of this corollary is based on the following lemma:

\begin{lemma}\label{blowuprate}
 Let $\gamma:[a,b]\times[0,T)\to\R^2$ be a solution of (\ref{1}) with $T<\infty$ is a time such that $\{\max_{[a,b]}\kappa^2(\cdot,t): t\in[0,T)\}$ is unbounded. Then we have that
 \begin{align*}
  \kappa_{max}^2(t)\defi\max_{[a,b]}\kappa^2(\cdot,t)\geq \frac{1}{4(T-t)} \ \ \forall t\in(0,T).
 \end{align*}
\end{lemma}

\begin{proof}
 A bound $ \max_{[a,b]}\kappa^2(\cdot,t) \geq \frac{1}{2(T-t)}$ was proved in \cite[Proposition~4.1]{MeinPaper} for a convex initial curve. We refine this proof for a general initial curve: We compute the evolution equation of $\kappa^2$ and estimate
\begin{align*}
 \pt \kappa^2 &= \ps^2 \kappa^2 - 2(\ps\kappa)^2 + 2\kappa^4 - 2 \kappa^3\bar\kappa\\
  &\leq \ps^2 \kappa^2 + 2\kappa^4 + 2\left(\max_{[a,b]}|\kappa|(\cdot,t)\right)^3|\bar\kappa|\\
  &\leq \ps^2 \kappa^2 + 4\left(\max_{[a,b]}|\kappa|(\cdot,t)\right)^4,
\end{align*}
where we used $-\max_{[a,b]}|\kappa|\leq\bar\kappa\leq\max_{[a,b]} |\kappa|$ in the last step. As $\kappa^2$ is $C^2$ the function $t\mapsto \kappa^2_{max}(t)$ is Lipschitz and hence differentiable almost everywhere. At a point of differentiability we can compute the time derivative as $\frac{d}{dt} \kappa^2_{max}(t) = \frac{\partial \kappa^2(p,t)}{\partial t}$, where $p\in[a,b]$ is a point where the maximum is attained. This approach is sometimes called ``Hamilton's trick''. It goes back to \cite{Hamiltontrick}. We get that 
\begin{align}\label{ungl2}
 \frac{d}{dt} \kappa^2_{max}(t) \leq \ps^2\kappa^2(p,t) + 4 \left(\kappa^2_{max}(t)\right)^2,
\end{align}
where $p\in[a,b]$ is a point where the maximum of $\kappa^2(\cdot,t)$ is attained.  We now prove that 
\begin{align}\label{psleq0}
 \ps^2\kappa^2(p,t)\leq0
\end{align}
holds for such a point $p\in[a,b]$. If $p\in(a,b)$, we simply have a maximum in the inner part of $[a,b]$. Thus, inequality (\ref{psleq0}) is clear. So we assume that $p=a$. \\
Case~$i)$: $\kappa(a,t)>0$: Then $\kappa(a,t)=\max_{[a,b]}\kappa(\cdot,t)$. So we have the inequality $\ps\kappa(a,t)\leq 0$. In \cite[Lemma~2.12]{MeinPaper} we proved by differentiating the boundary conditions that $\ps\kappa(a,t) = \left(\kappa(a,t)-\bar\kappa(t)\right) \kappa_\Sigma( \gamma(a,t))$ for all $t\in(0,T)$. In our specific situation we get that
\begin{align*}
 0 \geq \ps\kappa(a,t)= \left(\kappa(a,t)-\bar\kappa(t)\right) \kappa_\Sigma( \gamma(a,t)) \geq 0,
\end{align*}
where we used $\bar\kappa(t)\leq\max_{[a,b]}\kappa(\cdot,t) =  \kappa(a,t) $ in the last inequality. We hence get that $\ps\kappa (a,t)=0$ and therefore $\ps \kappa^2(a,t) =2\kappa(a,t)\ps\kappa(a,t)=0$. A positive sign of the second derivative $\ps^2\kappa^2(a,t)>0$ would now imply a strict local minimum of $\kappa^2(\cdot,t)$ in $a$, which is a contradiction. As a consequence we get that (\ref{psleq0}) is satisfied.\\
Case~$ii)$: $\kappa(a,t)<0$: In this case we know that $\kappa(a,t)=\min_{[a,b]}\kappa(\cdot,t)$. So we get that $\ps\kappa(a,t)\geq 0$ and
\begin{align*}
 0 \leq \ps\kappa(a,t)= \left(\kappa(a,t)-\bar\kappa(t)\right) \kappa_\Sigma( \gamma(a,t)) \leq 0
\end{align*}
because of $\bar\kappa(t)\geq \min_{[a,b]}\kappa(\cdot,t)= \kappa(a,t) $. Thus, we also have $\ps\kappa^2(a,t)=0$. As in the first case, we get that $\ps^2\kappa^2(a,t)\leq 0$.\\
Case~$iii)$: $\kappa(a,t)=0$: Here, we immediately get that $\ps\kappa^2(a,t)= 2\kappa(a,t)\ps\kappa(a,t)=0$. As in the other two cases, this implies $\ps^2\kappa^2(a,t)\leq 0$ because $a$ is a maximum point of $\kappa^2(\cdot,t)$.\\
If $p=b$, (\ref{psleq0}) follows analogously as $\ps\kappa(b,t)= -  \left(\kappa(b,t)-\bar\kappa(t)\right) \kappa_\Sigma( \gamma(b,t)) $ \cite[Lemma~2.12]{MeinPaper}.\\

We now use (\ref{ungl2}) and (\ref{psleq0}) and get
\begin{align*}
 -\frac{d}{dt} \left(\frac{1}{\kappa^2_{max}(t)}\right)\leq 4
\end{align*}
at all times $t\in(0,T)$ where $\kappa^2_{max}$ is differentiable. Integrating and using the existence of a sequence $t_j\to T$ such that $\kappa_{max}^2(t_j)\to\infty$ yields the result.
\end{proof}

\begin{definition}
 We keep the notation of a type I singularity as in the (classical) curve shortening flow: A singular time $T<\infty$ is \emph{of type I} if there is a constant $c>0$ such that
 \begin{align*}
  \max_{[a,b]}\kappa^2(\cdot,t)\leq \frac{c}{T-t} \  \ \ \forall t\in[0,T).
 \end{align*}
Otherwise, the singularity is of type II.
\end{definition}

\begin{proof}
 (of Corollary~\ref{typeII}).
 In \cite[Theorem~4.16]{MeinPaper}, the author proved that a \emph{convex} initial curve cannot develop a type I singularity in finite time if $|\bar\kappa|\leq c_2$ and $L(\gamma_t)\geq c_1>0$. We are able to generalize this result for general initial curves under the same bounds on the total curvature and on the length.\\
Almost all steps of the proof of Theorem~4.16 in \cite{MeinPaper} are already formulated for the general case, see Section~4 in \cite{MeinPaper}. We sketch the most important steps: Assume that the flow develops a singularity of type I in finite time. We do a parabolic rescaling
\begin{align*}
 \tilde\gamma_j(p,\tau)\defi Q_j \left(\gamma(p,\tfrac{\tau}{Q_j^2} + T) - x_0\right) \  \ \text{ for } (p,\tau)\in [a,b]\times [-Q_j^2T,0),
\end{align*}
 where $x_0\in\R^2$ is a ``blowup point'' of the flow, which means $t_j\to T$, $p_j\to p_0\in[a,b]$, $Q_j= |\kappa|(p_j,t_j) =\max_{[a,b]}|\kappa|(p,t_j)|\to\infty$, $\gamma(p_j,t_j)\to x_0$. Using the gradient estimates from Stahl \cite{StahlPaper1, StahlPaper2} we adapted the convergence procedure from \cite[Remark~4.22~(2)]{EckerBuch} to the area preserving flow. This is similar to the procedure in Theorem~\ref{convthm} (but it is not necessary to use integral estimates because $T<\infty$). We get smooth subconvergence (after reparametrization) to a limit flow $\gamma_\infty:I\times(-\infty,0)\to\R^2$, where $I$ is an interval containing $0$. \\
 Because of the $L^\infty$ bound on $\bar\kappa(t)$ the term $\bar\kappa_j(t)$ is scaled away in the limit. Thus, the limit flow satisfies $\pt \gamma_\infty=\kappa_\infty \nu_\infty$, it is an ancient solution of the curve shortening flow. 
 The lower bound on the length implies that each curve $\gamma_\infty(\cdot,t)$ has infinite length. If the singularity develops at the boundary then the curve $\gamma_\infty(\cdot,t)$ meets a straight line perpendicularly at the endpoint. We reflect it and can consider a complete, unbounded solution of the curve shortening flow. A monotonicity formula for the free boundary situation yields the key properties of the limit flow: Each curve $\gamma_\infty(\cdot,t)$ is proper and $\gamma_\infty$ is self-similarly shrinking, i.e.\ $\kappa_\infty(p,\tau)= \frac{\langle \gamma_\infty(p,\tau),\nu_\infty(p,\tau)\rangle}{2 \tau}$. \\
 For plane curves, all the self-similarly shrinking solutions are classified. It turns out that the curvature of these solutions does not change sign, see \cite{Halldorsson}. We get that $\gamma_\infty$ is one of the following:
 \begin{enumerate}
  \item The line $\R\times\{0\}$, \label{loe1}
  \item the shrinking spere $\mathbb S^1_{\sqrt{-2\tau}}$, where the curves can also be negatively oriented,\label{loe2}
  \item one of the closed ``Abresch-Langer curves'' \cite{AbreschLanger}, positively or negatively oriented,  \label{loe3}
  \item a curves whose image is dense in an annulus of $\R^2$.  \label{loe4}
 \end{enumerate}
The solutions \ref{loe2}), \ref{loe3}) and \ref{loe4}) are excluded because of the unbounded length and the properness of the curves. It remains to exclude \ref{loe1}): We rescaled at points of maximal curvature which implies for $\tau_j\defi -Q_j^2(T-t_j)$
\begin{align*}
 \kappa_{\tilde\gamma_j}(p_j,\tau_j)= \frac{1}{Q_j}\kappa\left(p_j,\tfrac{-Q_j^2(T-t_j)}{Q_j^2} + T\right) = \frac{1}{Q_j}\kappa(p_j,t_j)=1 \ \ \forall j\in\N.
\end{align*}
We reparametrize in the spatial component such that $\tilde\kappa_j(0,\tau_j)=1$ for all $j\in\N$. 
By the type I property we get that $$\tau_j = -\kappa^2(p_j,t_j)(T-t_j) \geq - \frac{c}{T-t_j} (T-t_j)=  - c>-\infty.$$ The blowup rate from Lemma~\ref{blowuprate} yields $$\tau_j = -\kappa^2(p_j,t_j)(T-t_j) \leq - \frac{1}{4}<0.$$
Thus, there is a time $\tau\in[-c,-\frac{1}{4}]$ such that $\kappa_\infty(0,\tau)=1$. This excludes the line as a limit flow.
\end{proof}

\begin{corollary} \label{corHam}
 Let $\gamma_0:\ab\to\R^2$ be a convex initial curve satisfying the conditions from Theorem~\ref{singu}. Then the ``Hamilton blow-up'' at $T_{max}$ yields either a grim reaper (we call this situation an ``inner singularity'') or half a grim reaper at a plane (a ``boundary singularity'').
\end{corollary}

\begin{proof}
 The situation of a finite type II singularity was treated in \cite[Section~6]{MeinPaper}. We repeat the important steps for the sake of completeness. We recall the ``Hamilton blow-up'' \cite{Hamilton2}: Define $T\defi T_{max}$. 
 For $j\in\mathbb N$ choose $t_j\in [0,T-\frac{1}{j}]$ and $p_j\in[a,b]$ such that
\begin{align*}
 |\kappa|^2(p_j,t_j) \left(T-\tfrac{1}{j} - t_j\right) = \max\left\{\left(|\kappa|^2(p,t) \left(T-\tfrac{1}{j} - t\right)\right) : t\in[0,T-\tfrac{1}{j}], p\in[a,b]\right\}.
\end{align*}
Then define $Q_j\defi |\kappa|(p_j,t_j)$ and 
\begin{align*}
\tilde\gamma_j(\cdot,\tau)\defi Q_j\left(\gamma(\cdot,\tfrac{\tau}{Q_j^2} + t_j) - \gamma(p_j,t_j)\right) \mbox{ for } \tau\in[-Q_j^2t_j,Q_j^2(T-t_j-\tfrac{1}{j})] \mbox{ on }[a,b].
\end{align*} 
 As the singularity is of type II, one can show certain properties of the rescaled flow. The most important ones are $\tilde\kappa_j(p_j,0)=0 \  \forall j$, $|\tilde\kappa_j| (\cdot,\tau)\leq 1 \ \forall j$ and 
 \begin{align*}
   \forall \epsilon>0 \ \forall \bar \tau>0 \ \exists j_0(\epsilon,\bar \tau)\in\mathbb N,\  \forall j\geq j_0: 
 & |\tilde\kappa_j|^2(p,\tau) \leq 1 + \epsilon \\
 &\forall \tau \in [-Q_{j_0}^2t_{j_0},\bar \tau], \forall p\in[a,b].
 \end{align*}
 Then there exist reparametrizations $\psi_j: I_j \to [a,b]$ with $|I_j|\to\infty$ ($j\to\infty$) such that a subsequence of the rescaled curves
\begin{align*}
 \gamma_j\defi\tilde\gamma_j(\psi_j,\cdot): I_j\times [-Q_j^2t_j,Q_j^2(T-t_j-\tfrac{1}{j})]\to\Rzwei
\end{align*}
converges locally smoothly to a limit flow $\tilde\gamma_\infty: \tilde I \times (-\infty,\infty)\to \Rzwei$ (where $\tilde I$ is an unbounded interval containing $0$). The proof of this subconvergence can be found in \cite[Proposition~6.2, Proposition~4.7]{MeinPaper}. It is again similar to the proofs of Theorem~\ref{convthm} and Corollary~\ref{typeII}.\\
The limit flow $\tilde\gamma_\infty$ is a smooth solution of the curve shortening flow and satisfies $0<\tilde\kappa_\infty\leq 1$ everywhere and $\tilde\kappa_\infty = 1$ at least at one point. If $\tilde M_\tau^\infty\defi\tilde\gamma_\infty (\tilde I,\tau)$ has a boundary, then $\partial\tilde M^\infty_\tau\subset \Sigma_\infty$, where $\Sigma_\infty $ is a line through $ 0\in\Rzwei$, and $\langle \tilde\nu_\infty, \nu _{ \Sigma_\infty}\rangle = 0$ on $\partial\tilde M_\infty$. 
By reflecting at the line $\Sigma_\infty$ one gets an eternal solution of the curve shortening flow with bounded curvature where the maximal curvature is attained at least at one point. Due to \cite[Theorem~1.3]{Hamilton}, 
  the limit flow must be a translating solution, and the only translating solution in the case of curves is the ``grim reaper'' which is the flow of curves given by $x=-\log\cos y + \tau$ for $y\in (-\frac{\pi}{2},\frac{\pi}{2})$.
  In the situation where the limit flow does have a boundary it must be ``half the grim reaper'' at $ \Sigma_\infty$ because the grim reaper  has only one symmetry axis.
\end{proof}
In \cite{Chou} and \cite{EscherIto} the blowup-rate at the singularity was characterized for the $L^2$-norm of the curvature, and not for the $C^0$-norm as above. This $L^2$-rate can also be proved for the free boundary setting:
\begin{proposition} 
 Let $\gamma:[a,b]\times[0,T_{max})\to\Rzwei$ be a solution of (\ref{1}) with $T_{max}<\infty$ and $|\bar\kappa|\leq c<\infty$. Then there is a constant $C>0$ and a sequence of times $t_k\to T_{max}$ such that
 $$\int|\kappa(\cdot,t_k)|^2d s_{t_k} \geq C(T_{max}-t_k)^{-\frac{1}{2}}$$
\end{proposition}
\begin{proof}
The proof is due to \cite[Proposition~A]{Chou} and \cite[Proposition~5]{EscherIto}. Since $T_{max}<\infty$ we have that $\left\{\int\kappa^2 ds: t\in[0,T_{max})\right\}$ is unbounded. As it was pointed out in \cite[Proof of Proposition~5]{EscherIto}, this comes from the fact that the proof of the short time existence only depends on the $C^{1,\alpha}$-norm of the initial data for all $\alpha\in (0,1)$. For the Neumann boundary condition setting the estimates behind this argument can be found in \cite[Lemma~5.3.2]{MeineDiss}.\\
In order to follow the proof of \cite[Proposition~5]{EscherIto} we only have show that $$\frac{d}{dt} E(t)\leq C\left(E(t) + E(t)^3\right)$$ for $E(t):= \int(\kappa-\bar\kappa)^2 ds$. In \cite[Corollary~7.4]{MeinPaper}, the inequality $$\frac{d}{dt} E(t)\leq C\left(E(t) + E(t)^{\frac{5}{3}} + E(t)^3\right)$$ is proved under the condition $|\bar\kappa|\leq c<\infty$. We have 
\begin{align*}
 E^{\frac{5}{3}} \leq \begin{Bmatrix}
                       E, & \text{ if } 0\leq E\leq 1\\
                       E^3, & \hspace{-0.7cm}\text{ if } 1\leq E
                      \end{Bmatrix}
\leq E^3 + E.
\end{align*}
This was also used in \cite[Proof of Proposition~5]{EscherIto}. 
\end{proof}
The following corollary is immediate.
\begin{corollary}
 Under the conditions of Theorem~\ref{singu} there is a sequence of times $t_k\to T_{max}<\infty$ such that $\int|\kappa(\cdot,t_k)|^2d s_{t_k} \geq C(T_{max}-t_k)^{-\frac{1}{2}}$.
\end{corollary}

\section{Examples} \label{sec3}

It remains to show that there are curves that satisfy the conditions from Theorem~\ref{singu} or Corollary~\ref{typeII}. 

\begin{example}[Example One]
Let us consider a convex curve $\Sigma$ that almost looks like a circle with $d_\Sigma> 2\pi$. Then one can construct an initial curve $\gamma_0:[0,1]\to\R^2$ with $L_0<\frac{4}{3}\pi$, $l=2$ and $A(\gamma_0 + \sigma_0)>\frac{\pi}{2}$. An example is drawn in Figure~\ref{Bild16}. Note that $\sigma_0$ is the connection of $\gamma_0(1)$ and $\gamma_0(0)$ along $\Sigma$ that is visible in the picture. We check the isoperimetric quotient of that initial curve and compare it to the conditions of Theorem~\ref{singu}:
\begin{align*}
 \frac{L_0^2}{A(\gamma_0 + \sigma_0)}< \frac{2}{\pi} \left(\frac{4}{3}\pi\right)^2 = 2 \frac{16}{9} \pi< \frac{9}{2} \pi.
\end{align*}

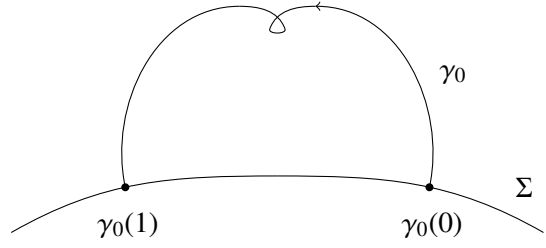
\begin{figure}
	  \begin{center}
	\begin{tikzpicture}
 \draw[out=30,in=180] (-2,-1.1) to (1.5,-0.35);
 \draw[out=0, in=150] (1.5,-0.35) to (5,-1.1);          
          \coordinate[label =right: $\Sigma$] (A) at (4.5,-0.5);
             \draw[->,out=78,in=0] (3.5,-0.5) to (2,1.9);
             \coordinate[label =right: ${\gamma_0(0)}$] (C) at (3,-1);
             \fill (3.5,-0.5) circle (1.5pt);
                      \draw[out=180,in=70] (2,1.9) to (1.4,1.6);
                       \draw[out=-110,in=-70] (1.4,1.6) to (1.6,1.6);
                       \draw[out=110,in=0]  (1.6,1.6) to (1,1.9);
                          \draw[out=180,in=102] (1,1.9) to (-0.5,-0.5);
                           \coordinate[label =right: ${\gamma_0(1)}$] (D) at (-1,-1); 
                           \fill (-0.5,-0.5) circle (1.5pt);
                          \coordinate[label =right: ${\gamma_0}$] (B) at (3.5,1);
      \end{tikzpicture}
	  \end{center}
	   \caption{An initial curve, where the flow develops a singularity in finite time, see Example~One.} \label{Bild16} 
\end{figure}

Thus, this curve develops a type II singularity in finite time. This is somehow not surprising as it was shown in \cite[Proposition~9]{EscherIto} that a curve looking like the described $\gamma_0$ but closed on the ``lower part'' (a so-called ``lima\c{c}on'') develops a singularity in finite time under the area preserving curve shortening flow \emph{without boundary}. And the ``lima\c{c}on'' is the classical example where the curve shortening flow (without boundary) develops a type II singularity \cite{Angenent}. These type II singularities are usually expected when there is a self-intersection. \\
But there are examples satisfying the conditions from Case~\ref{case11}) in Theorem~\ref{singu} that seem to behave differently, see Example~Two. 
\end{example}

\begin{example}[Example Two]

We construct $\gamma_0:[0,1]\to\R^2$ as shown in Figure~\ref{Bild1}. Again, $\sigma_0$ is the connection of $\gamma_0(1)$ to $\gamma_0(0)$ along $\Sigma$. As in the first example we have that $l=2$. We construct $\gamma_0$ such that $L_0<d_\Sigma$, $L_0< \frac{4}{3}\pi$ and $A(\gamma_0 + \sigma_0)>\frac{\pi}{2}$. We conclude again
\begin{align*}
 \frac{L_0^2}{A(\gamma_0 + \sigma_0)  }< \frac{2}{\pi} \left(\frac{4}{3}\pi\right)^2 = 2 \frac{16}{9} \pi< \frac{9}{2} \pi.
\end{align*}
For this particular $\gamma_0$ we conjecture that the curves stay embedded under the flow (\ref{1}) and that the type II singularity forms at  the boundary. 

\begin{figure} 
	  \begin{center}
	  \scalebox{0.7}{
	\begin{tikzpicture}
 \draw[out=20,in=180] (-3.5,-1.3) to (1.5,-0.35);
 \draw[out=0, in=160] (1.5,-0.35) to (6.5,-1.3);    
          \coordinate[label =right: $\Sigma$] (A) at (1.5,-0.8);
          \draw[->,out=97, in=0] (-1.2,-0.6) to (-1.6,-0.3); 
            \coordinate[label =below: ${\gamma_0(0)}$] (C) at (-1.2,-0.8);
             \fill (-1.2,-0.6) circle (1.5pt);
           \draw[out=180, in=90] (-1.6,-0.3) to (-2.2,-1);  
                    \draw[out=270, in=180]  (-2.2,-1) to (1.5,-5);  
                     \draw[out=0, in=270]   (1.5,-5) to (5.2,-1); 
                      \draw[out=90, in=0]   (5.2,-1) to (4.6,-0.3); 
                       \draw[out=180, in=83]    (4.6,-0.3) to (4.2,-0.6); 
                           \coordinate[label =below: ${\gamma_0(1)}$] (D) at (4.2,-0.8);
             \fill (4.2,-0.6) circle (1.5pt);
                          \coordinate[label =right: ${\gamma_0}$] (B) at (3.5,-3.3); 
                          
      \end{tikzpicture}
      }
	  \end{center}
	   \caption{Another initial curve, where the flow develops a singularity in finite time, see Example~Two.} \label{Bild1} 
\end{figure}
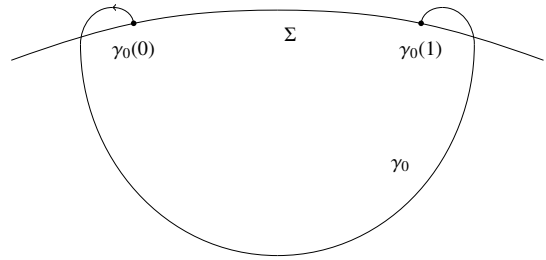

\end{example}

\begin{example}[Example Three]
The conditions of Theorem~\ref{singu}, Case~\ref{case21}) are satisfied by a curve $\gamma_0:[0,1]\to\R^2$ as shown in Figure~\ref{Bild2}. We choose $G_\Sigma$ big enough such that $L_0< d_\Sigma$. We have that $\kappa_0>0$ and $l=2$. We have constructed $\gamma_0$ in such a way that $A(\gamma_0 +\sigma_0)<0$. By Theorem~\ref{singu} we get a singularity in finite time that is of type~II.
\begin{figure} 
	  \begin{center}
	\begin{tikzpicture}
	\draw[out=10, in=180] (-3,-0.3) to (0,0);
	\draw[out=0, in=170]  (0,0) to (3,-0.3);
	\draw[->,out=82, in=70]  (1,-0.03) to (0,-0.9);
	  \coordinate[label =right: ${\gamma_0(0)}$] (C) at (1.1,0.2);
             \fill (1,-0.03) circle (1.2pt);
	\draw[out=-110, in=180]  (0,-0.9) to (-0.0,-1.5);
	\draw[out=0, in=-40]  (-0.0,-1.5) to (-0.8,-0.5);
	\draw[out=140, in=100]   (-0.8,-0.5) to (-1.6,-0.08);
	 \coordinate[label =left: ${\gamma_0(1)}$] (C) at (-1.6,0.2);
	  \fill (-1.6,-0.08) circle (1.2pt);
	\coordinate[label =right: $\Sigma$ ] (B) at (2.5,-0.8);
	  \coordinate[label =right: ${\gamma_0}$] (A) at (0.4,-0.7);
      \end{tikzpicture}
 \end{center}
  \caption{Another initial curve, where the flow develops a singularity in finite time, see Example~Three.} \label{Bild2} 
  \end{figure}
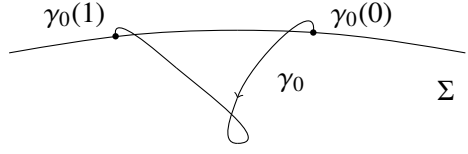
\end{example}

\begin{example}[Example Four]
 As Theorem~\ref{singu} gives the existence of singularities also for non-convex curves, we provide such an example, see Figure~\ref{Bild3}. The initial curve $\gamma_0:[0,1]\to\R^2$ satisfies $\int_{\gamma_0}\kappa ds \in (-2\pi,0)$ but $A(\gamma_0 + \sigma_0)>0$. After changing the orientation Case~\ref{case21}), Theorem~\ref{singu} applies, and the flow develops a singularity in finite time.
 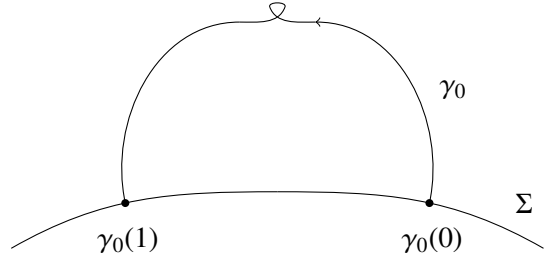
\begin{figure}
	  \begin{center}
	\begin{tikzpicture}
 \draw[out=30,in=180] (-2,-1.1) to (1.5,-0.35);
 \draw[out=0, in=150] (1.5,-0.35) to (5,-1.1);          
          \coordinate[label =right: $\Sigma$] (A) at (4.5,-0.5);
             \draw[->,out=78,in=0] (3.5,-0.5) to (2,1.9);
             \coordinate[label =right: ${\gamma_0(0)}$] (C) at (3,-1);
             \fill (3.5,-0.5) circle (1.5pt);
                      \draw[out=180,in=-70] (2,1.9) to (1.4,2.1);
                       \draw[out=110,in=70] (1.4,2.1) to (1.6,2.1);
                       \draw[out=-110,in=0]  (1.6,2.1) to (1,1.9);
                          \draw[out=180,in=102] (1,1.9) to (-0.5,-0.5);
                           \coordinate[label =right: ${\gamma_0(1)}$] (D) at (-1,-1); 
                           \fill (-0.5,-0.5) circle (1.5pt);
                          \coordinate[label =right: ${\gamma_0}$] (B) at (3.5,1);
      \end{tikzpicture}
	  \end{center}
	   \caption{A non-convex initial curve, where the flow develops a singularity in finite time, see Example~Four.} \label{Bild3} 
\end{figure}
 
\end{example}

\section{The area preserving curve shortening flow at a straight line} \label{sec4}
In this section, we consider the area preserving curve shortening flow (APCSF) at a straight line. 
We prove that there are initial curves that develop a singularity in finite time. 
The situation is somehow easier than in the previous section. The strategy is to reflect the curves at the line and to use the results from \cite{EscherIto} for the closed case. At first we have to specify some notation for the case that $\Sigma$ is a straight line.

\begin{definition}
Consider the map $f:
s\mapsto (-s,0) \in\R^2$, $s\in(-\infty,\infty)$. The map $f$ parametrizes the line $\Sigma\defi \{(x,y)\in\R^2:x\in\R,y=0\}$.\\
 A smooth, regular curve $\gamma_0:[a,b]\to\Rzwei$ is called \emph{initial curve} if it satisfies the conditions
\begin{align*}
 \gamma_0(a),\gamma_0(b)&\in\Sigma\\
\tau_0(a) &= e_2, \\
\tau_0(b) &= -e_2,
\end{align*}
  where $\tau_0=\ps \gamma_0$ is the tangent of $\gamma_0$ and $e_2= (0,1)\in\R^2$ is the second standard vector in $\R^2$.
\end{definition}

    \begin{definition}
 Let $f:[a,b]\to\Rzwei$ be a piecewise smooth, regular and closed curve. The number
\begin{align*}
\ind(f)\defi n(\partial_pf,0) \in \mathbb Z
\end{align*}
is called the \emph{index} (or \emph{turning number}) of $f$. Here, $n(\partial_p f,0)$ denotes the winding number of the curve $\partial_p f:[a,b]\to\Rzwei$ with respect to $0\in\Rzwei$.
\end{definition}

 \begin{theorem} \label{indexformel}
 Let $f$ be a piecewise smooth, regular and closed curve, defined on intervals $[a_j,b_j]$, $j=1,\dots,k$, and with exterior angles $\alpha_j$, $j=1,\dots,k$. Then
\begin{align*}
 \ind(f)=\frac{1}{2\pi}\sum_{j=0}^k\int_{a_j}^{b_j}\kappa_f\de s_f + \frac{1}{2\pi} \sum_{j=0}^k\alpha_j  \ \ \in\mathbb Z.
\end{align*}
\end{theorem}
\begin{proof}
 See \cite[Theorem 2.1.6]{Klingenberg}. 
\end{proof}

\begin{lemma} \label{mgeq1}
  Let $\gamma_0:\ab\to\R^2$ be an initial curve. Reflect the curve $\gamma_0$ at the line $\Sigma$ into the lower half space of $\R^2$. Then the resulting closed curve $\delta_0$ is a $C^2$ curve with $\ind(\delta_0)=:m \in\Z$. The number $m$ is odd. 
\end{lemma}

\begin{proof}
 As the curves meet $\Sigma$ perpendicularly and because of reflection at a line the reflected curves are $C^2$. We treat two cases:\\ 
 Case 1: $f^{-1}(\gamma_0(a))\leq f^{-1}(\gamma_0(b))$. \\
 Then consider $\gamma_0 + \sigma_0$ where $\sigma_0$ is the line segment from $\gamma_0(b)$ to $\gamma_0(a)$. The exterior angles at the points where $\tau_0$ is not continuous are $+\frac{\pi}{2}$ (or $+\pi$ if $\gamma_0(a)=\gamma_0(b)$). Thus, we have $l\defi\ind(\gamma_0 +\sigma_0)= \frac{1}{2\pi}\left(\int_{\gamma_0}\kappa d s + \pi\right) \in\Z$ or equivalently $\int_{\gamma_0}\kappa ds = 2\pi l - \pi$. After reflecting we get $\ind(\delta_0)= \frac{2\int_{\gamma_0}\kappa d s}{2\pi} =  2l - 1$. \\
 Case 2: $f^{-1}(\gamma_0(a)) > f^{-1}(\gamma_0(b))$.\\
 We denote by $\sigma_0$ the line segment from $\gamma_0(b)$ to $\gamma_0(a)$. Note that this is oriented in the opposite direction compared to $f$. 
 Now the exterior angles of $\gamma_0 + \sigma_0$ are $-\frac{\pi}{2}$. This implies $\int\kappa ds = 2\pi l +\pi$ for $l\in\Z$. By reflection we conclude $ \ind(\delta_0)= \frac{2\int_{\gamma_0}\kappa d s}{2\pi} =  2l +1.$
\end{proof}

\begin{remark}
 The APCSF preserves the reflection symmetry with respect to the $x$-axis. It hence does not matter whether we start at the straight line the APCSF with Neumann free boundary conditions and then reflect at $\Sigma$ or if we reflect at first and then consider the APCSF for closed curves. Thus, we recover the APCSF with Neumann free boundary conditions from the flow of the closed curves. 
\end{remark}

\begin{proposition}\label{propsingu}
  Let $\gamma_0:\ab\to\R^2$ be an initial curve. Reflect $\gamma_0$ at $\Sigma$ and denote the closed curve by $\delta_0$. Choose the orientation of $\delta_0$ such that $\ind(\delta_0)=:m\geq 0$. Lemma~\ref{mgeq1} shows that $m$ is odd.\\
  Then the area preserving curve shortening flow with Neumann free boundary conditions at the line $\Sigma$ develops a  singularity in finite time if one of the following conditions is satisfied:
  \begin{enumerate}
   \item Either $A(\delta_0)<0$.
   \item Or $m\geq 3$ and $L(\delta_0)^2 < 4\pi m A(\delta_0)$.
  \end{enumerate}

\end{proposition}

\begin{proof}
 We use Lemma~\ref{mgeq1} to get that $m$ is odd, so $m\geq 1$ is always satisfied. Use \cite[Proposition~9]{EscherIto} for the flow of the reflected curve to get that that $T_{max}<\infty$. 
\end{proof}

\begin{corollary}
The finite time singularity appearing in Proposition~\ref{propsingu} is of type~II.
\end{corollary}

\begin{proof}
 Denote by $\delta_t$, $t\in[0,T_{max})$, the closed curves and with $\gamma_t$, $t\in[0,T_{max})$, the curves with boundary. By the isoperimetric inequality for $\delta_t$ we get that $L(\delta_t)^2\geq 4\pi |A(\delta_t)| = 4\pi |A(\delta_0)|$. This implies $L(\gamma_t)^2\geq \pi |A(\delta_0)|>0$. Thus the length is bounded from below uniformly in $t$. We have that $2\int_{\gamma_0}\kappa d s = \int_{\delta_0}\kappa d s=2\pi m\in\Z$. Continuity yields $\int_{\gamma_t}\kappa d s= \pi m$ for all $t\in[0,T_{max})$. Thus $|\bar\kappa_{\gamma_t}(t)|\leq c_2<\infty$ uniformly in $t$.
 A blowup argument as in \cite[Theorem~4.16]{MeinPaper} or as in the proof of Corollary~\ref{typeII} implies that the singularity is of type II.
\end{proof}

\bibliographystyle{plain}
\bibliography{Lit}

\end{document}